\documentclass{aptpub}
\usepackage{graphicx}
\usepackage{amsmath,dsfont}
\usepackage{amsfonts}
\usepackage{amssymb}
\usepackage{setspace}

\usepackage[pdftex,colorlinks]{hyperref}
\usepackage{natbib}

\renewcommand{\crnotice}{}

\newcommand{\bfe}{\mathbf{e}}

\newcommand{\bfF}{\mathbf{F}}
\newcommand{\bfW}{\mathbf{W}}
\newcommand{\ir}{\Bbb{R}}

\newcommand{\bbZ}{\Bbb{Z}}
\newcommand{\bbN}{\Bbb{N}}
\newcommand{\norm}[1]{|{#1}|}
\newcommand{\Pro}{{\mathbf P}}
\newcommand{\md}{\,\text{d}}
\newcommand{\myd}{\text{d}}
\newcommand{\SINR}{\text{SINR}}
\newcommand{\SNR}{\text{SNR}}
\renewcommand{\E}{\mathbf{E}}
\newcommand{\ind}{\mathds{1}}
\newcommand{\Gsinr}{{\mathbb G}} 
\newcommand{\calH}{{\cal H}}

\newcommand{\calL}{{\cal L}}
\newcommand{\LD}{L}
\newcommand{\Trial}{{\cal T}}
\newcommand{\ld}{\ell}
\newcommand{\qed}{\hfill $\Box$}

\authornames{F. Baccelli, B. Blaszczyszyn, O. Mirsadeghi}
\shorttitle{Optimal Paths on the Space--Time SINR Random Graphs}

\begin{document}


\title{Optimal Paths on the Space--Time SINR Random
  Graph}\\[2ex]
\authorone[INRIA/ENS]{Fran\c cois Baccelli\vspace{-5ex}}
\addressone{ENS DI TREC, 45 rue d'ulm, 75230 Paris, FRANCE.}
\noindent\authortwo[INRIA/ENS and Math. Inst. University of
Wroc\l aw]{Bart\L omiej  B\L aszczyszyn\vspace{-5ex}}
\addresstwo{ENS DI TREC, 45 rue d'ulm, 75230 Paris, FRANCE.}
\noindent\authorthree[INRIA/ENS and Sharif University of
Technology]{Mir Omid Haji Mirsadeghi}
\addressthree{ENS  DI TREC, 45 rue d'ulm, 75230 Paris, FRANCE. The
work of this author is part of a joint PhD programme with a co-advising
by Prof. Amir Daneshgar in Tehran}
\vspace{-4ex}

\begin{abstract}
We analyze a class of 
Signal-to-Interference-and-Noise-Ratio 
(SINR) random graphs. These random graphs arise in the modeling
packet transmissions in wireless networks. 
In contrast to previous studies on the SINR graphs, we consider 
both a space and a time dimension. The spatial aspect originates
from the random locations of the network nodes in the Euclidean
plane. The time aspect stems from the random transmission policy  
followed by each network node and from the time variations of the
wireless channel characteristics.
The combination of these random space and time aspects  leads to 
fluctuations of the SINR experienced by the wireless channels,
which in turn determine the progression of packets in space and
time in such a network.  

This paper studies optimal paths in such wireless networks in 
terms of first passage percolation on this random graph. We 
establish both ``positive'' and ``negative'' results on the associated
time constant. The latter determines the asymptotics of the
minimum delay required by a packet to progress from a source node to
a destination node when the Euclidean distance between the two tends
to infinity. The main negative result states that this time constant is
infinite on the random graph associated with a Poisson point
process under natural assumptions on the wireless channels.  
The main positive result states that when adding a periodic node
infrastructure of arbitrarily small intensity to the Poisson point
process, the time constant is positive and finite.
\end{abstract}
\keywords{Poisson point process, random graph, first passage
  percolation, shot-noise process, SINR}
\ams{60D05, 05C80}{90C27, 60G55}

\section{Introduction}
There is a rich literature on random graphs generated
over a random point process. These graphs are often
motivated by physical, biological or social networks. 
Many interesting large scale properties of these
networks related to connectivity have been studied in terms 
of the percolation of the associated graphs.  An early example of 
such a study can be found in~\cite{Gilbert61}
where the connectivity of large networks
was defined as the supercritical phase in what is today called
the continuum (Boolean) percolation model.
More recently, a random SINR graph model for wireless networks
was studied with the same perspective in~\cite{Dousse_etal_TON,Dousse_etal}.

The routing, and more precisely, the speed of delivery of information
in networks is another example of problems, which motivated the 
study of the random graphs. The main object in this context 
is the evaluation of the so called {\em time constant},
which gives the asymptotic behavior of the number of edges (hops)
in the paths (optimal or produced by some particular routing protocol)
joining two given nodes in function of the (Euclidean) distance between
these nodes, when this distance tends to infinity.
In the case of a shortest (in terms of the number of hops)
path, this problem is usually called the {\em first passage percolation
problem} and was originally stated by Broadbent and Hammersley
in~\cite{Broad_Hamm57} to study the spread of fluid in a porous medium. 
More recently, in~\cite{Kleinberg,BordenaveAaP}, such time constants
were studied on so called {\em small word graphs}, motivated by routing
in certain social networks, where any two given nodes are joined by an
edge independently with a probability that decays as some power function
with the Euclidean distance between them. The complete graph on a Poisson p.p. 
with ``nearest neighbor'' routing policy was studied in
this context in~\cite{RST}. The first passage percolation problem on the 
Poisson-Delaunay graph was considered in~\cite{VahidiAsl_Wierman90,Pimentel2006}. 
In the case of graphs whose edges are marked by some weights,
one can extend the notion of time constant by studying the sum of the edge
weights. First passage percolation on the complete Poisson p.p. graph,
with weights proportional to some power of the distance between the nodes
was studied in~\cite{HowardNewman1997}.

The present paper focuses on the speed of delivery of information
in SINR graphs. In contrast to previous studies of this subject, in particular
to~\cite{Dousse_etal_TON,Dousse_etal}, we consider graphs with 
{\em space-time} vertexes. This new model is motivated by multihop routing
protocols used in wireless ad-hoc networks. In this framework, the random point
process on the plane describes the locations of the users of 
an ad-hoc network and the discrete time dimension corresponds to 
successive time slots in which these nodes exchange information (here packets).
As in~\cite{allerton} we assume the spatial Aloha policy to decide
which node transmit at a given time slot. We also assume some space-time
fading model (already used e.g. in~\cite{compj}) to describe the variability 
of the wireless channel conditions (see e.g.~\cite{TseViswanath2005}).
In this {\em space-time SINR graph}, a directed edge represents 
the feasibility of the wireless transmission between two given network
nodes at a given time.
More precisely, the direct transmission of a packet is succeeds
between two nodes in a given time slot if the
ratio of the power of the signal between these nodes 
to the interference and noise at the receiver is larger than a
threshold at this time slot. This definition has an information
theoretic basis (see e.g. \cite{TseViswanath2005}). It is
rigorously defined below using some power path-loss model
and an associated shot-noise model representing the interference. 

We study various problems on this random graph including the law
of its in- and out-degree, the number of paths originating
from (or terminating at) a typical node or its connectedness. 
The most important results bear on the first passage percolation problem
in this graph. In the case of Poisson p.p. for the node locations, we show 
that the time constant is infinite. We then show that when adding a
periodic node infrastructure of arbitrarily small intensity to the Poisson
point process, the time constant is positive and finite.
These results lead to bounds on the delays in ad-hoc networks which
hold for all routing algorithms. This  subject, or more generally, 
the question of the speed of the delivery of information in large 
wireless ad-hoc networks currently receives a lot of attention in
the engineering literature see e.g.~\cite{net:Ganti09spaswin,JMMR2009}.

The paper is organized as follows.
In Section~\ref{sec:oppstomo} we introduce the space-time SINR graph model.
The results are presented in Section~\ref{s.results},
Most of the proofs are deferred to Section~\ref{s.proofs}.
Some implications on routing in ad-hoc networks
are presented in Section~\ref{s.remarks}.

\section{The Model}
\label{sec:oppstomo}
\subsection{Probabilistic Assumptions}
\label{ss.Probab.assumptions}
Throughout the paper we consider a simple, stationary, independently
marked (i.m.) point process (p.p.)
$\widetilde\Phi=\{(X_i,\bfe_i,\bfF_i,\bfW_i)\}$
with finite,  positive intensity $\lambda$ on $\ir^2$. In this model,
\begin{itemize}
\item $\Phi=\{X_i\}$ denotes the locations of the network nodes on the
  plane $\ir^2$. The following three cases regarding the distribution
  of $\Phi$ will be considered:
\begin{itemize}
\item[{\bf General p.p.}:] $\Phi$ is a general (stationary, non-null, with finite intensity) p.p.,
\item[{\bf Poisson p.p.}:] $\Phi$ is a Poisson p.p.,
\item[{\bf Poisson$+$Grid p.p.}:]  $\Phi=\Phi_M+\Phi_{G}$ is the superposition of  two
  independent p.p.s; where $\Phi_M$ denotes a stationary Poisson
  p.p. with finite, non-null intensity $\lambda_M$ and 
$\Phi_{G}=s\bbZ^2+U_G$ a stationary, periodic p.p., whose nodes
constitute a square grid with edge length $s$,  randomly shifted by
the vector $U_G$ that is uniformly distributed in $[0,s]^2$
(this makes $\Phi_G$  stationary). Note that the intensity of $\Phi_G$ is 
$\lambda_G=1/s^2$.
\end{itemize}
\item  $\bfe_i=\{e_i(n)\}_{\in n\in\bbZ}$, where
  $\bbZ=\{\ldots,-1,0,1,\ldots\}$ denotes integers;
the variables $\{e_i(n):i,n\}$ are 
i.i.d. (in $n$ and $i$) Bernoulli random variables (r.v.s) with
$\Pro\{\, e=1\,\}=1-\Pro\{\,e=0\,\}=p$, where  $e$  denotes the
generic r.v. for this family. We always assume $0<p<1$.
The variable $e_i(n)$ represents the {\em medium access indicator}
of node $X_i$ at time $n$; it says whether the node transmits (case $e=1$)
or not at time $n$.
\item  $\bfW_i=\{W_i(n)\}_n$; $\{W_i(n):i,n\}$  is
a  family of   non-negative i.i.d. r.v.s with some arbitrary distribution. 
The variable $W_i(n)$ represents the power of the  {\em thermal noise}
at node $X_i$ and at time $n$. Let $W$ denote the generic r.v. for
this family.
\item $\bfF_i=\{F_{i,j}(n)\}_{j,n}$; $\{F_{i,j}(n):i,j,n\}$ is a family of
non-negative i.i.d. r.v.s. The variable $F_{i,j}(n)$ represents
the quality of the radio channel (also called {\em fading}) 
from node $X_i\in \Phi$ to node $X_j\in \Phi$ at time $n$.
The following two cases regarding the distribution of $F$ (denoting the
generic random variable for this family) will be considered:
\begin{itemize}
\item[{\bf General fading}:] when $F$ has some arbitrary distribution
  with finite mean.
\item[{\bf Exponential fading}:] when $F$ has  {\em exponential}
distribution with mean $1/\mu$.~(${}^*$)\footnote{(${}^*$)~In wireless signal
propagation models, the exponential distribution
appears naturally as the square power of the norm 
of a complex random variable, whose real and imaginary components are
i.i.d. Gaussian. I this case one often speaks about {\em Rayleigh
fading model} because  the norm (absolute value) of such a complex random
variable is Rayleigh-distributed; 
see e.g.~\cite[p.~50 and~501]{TseViswanath2005}.} 
\end{itemize}
\end{itemize}
To complete the probabilistic description of the model we assume that,
given $\Phi$, the random elements $\{\bfe_i\}_i, \{\bfW_i\}_i$ and $\{\bfF_i\}_i$
are independent. For more on this framework, which is classical, see
e.g. or \cite{bbm06,JSAC,compj}.

Our stationary i.m. p.p. $\tilde \Phi$ is considered on some
probability space with probability $\Pro$. We will denote by
$\Pro^0$ the Palm probability with respect
to $\Phi$; see~\cite[Ch.13]{DVJII2008}. 
Recall that it can be interpreted as the conditional
probability given $\Phi$ has a point at the origin  $0$ of the plane.
We will denote this point (considered under $\Pro^0$) by $X_0$
and call it the {\em typical node}.
Under $\Pro^0$ $\tilde\Phi$ is also an i.m. p.p. with marks distributed as in
the original law.  Moreover, in the case of Poisson p.p.s,
the distribution of $\Phi$ under $\Pro^0$ is equal to the distribution
of $\Phi\cup\{X_0=0\}$ under the stationary probability $\Pro$
(cf. the Slivnyak-Mecke Theorem~\cite[p.281]{DVJII2008}).

\subsection{SINR Marks}
\label{ss.SINR}
Given the i.m.p.p. $\tilde\Phi$ described above, we construct another 
family of random variables $\{\SINR_{ij}(n):i,j,n\}$, which will 
be interpreted as the SINR observed in the channel form $X_i\in\Phi$ to
$X_j\in\Phi$ at time $n$.  These variables, which have an information
theoretic background, will be used to assess the
success of transmissions. For defining these variables, we give ourselves
some non-decreasing function $l:\ir^+=\{t:t\ge0\} \to \ir^+$
that we call the {\em path-loss function}.
A special example considered in this paper (and commonly accepted in
the wireless communication context) is  
\begin{equation}\label{e.path-loss}
l(r)=(Ar)^\beta\qquad  \text{with some $A>0$ and $\beta>2$}\,.
\end{equation}
Denote by $\Phi^1(n)=\{X_i: e_i(n)=1\}$ the point process
of {\em transmitters} at time slot $n$ and by $\Phi^0(n)=\{X_i: e_i(n)=0\}$
that of (potential) {\em receivers}.
For a given pair receiver $X_j\in\Phi^0$ and transmitter $X_i\in\Phi^1(n)$, 
we will assume that $X_j$ receives a signal from $X_i$
with power $F_{i,j}(n)/l(\norm{X_j-X_i})$ at time $n$.
Node $X_j$ also receives signals from {\em other} transmitters
$X_k\in\Phi^1(n)$, $X_k\not=X_i$ at time $n$. 
The total received power is equal to 
$$ I_{i,j}(n)=\sum_{X_k\in \Phi^1(n)\setminus\{X_i\}} F_{k,j}(n)/l(\norm{X_k-X_j})\, .$$
Let also
$$I_{j}(n)=\sum_{X_k\in \Phi^1(n)\setminus\{X_j\}} F_{k,j}(n)/l(\norm{X_k-X_j})\,.$$
Both $I_{i,j}(n)$ and $I_{j}(n)$ are {\em  shot-noise}
r.v.s generated by  $\Phi^1(n)$, the fading marks and the path-loss
function. The are infinite sums of non-negative r.v.s. In order to
check whether these r.v.s  are a.s. finite, one can
use the Campbell-Little-Mecke formula 
(Campbell for short; cf.~\cite[Prop.~13.3.II]{DVJII2008}), which implies that
\begin{equation}\label{e.mean_sn}
\E^0\Bigl[\sum_{X_k\in \Phi^1(n), \norm{X_k}>\epsilon}
F_{k,0}(n)/l(\norm{X_k})\Bigr]=p\E[F]
\int_{\ir^2\setminus[0,\epsilon]^2}1/l(\norm{x})\breve M_{[2]}(\myd x)\,,
\end{equation}
where $\breve M_2(\cdot)$ is the {\em reduced second order moment measure}
of $\Phi$ (cf~\cite[p.~238]{DVJII2008}).
In what follows, we will always tacitly assume that 
$l(\cdot), \Phi$ are such that the integral in the right-hand side
of~(\ref{e.mean_sn}) is finite for some $\epsilon\ge0$,
which implies that $I_0(n)$ is almost surely (a.s.) finite under
$\Pro^0$ for all $n$ as well as all  $I_{j}(n), I_{i,j}(n)$ under $\Pro$. 
If $\Phi$ is the homogeneous Poisson p.p., we have
$\breve M_{[2]}(\myd x)=\lambda \md x$ and it is easy see that the
we have finiteness for $l(\cdot)$ given by~(\ref{e.path-loss})
for all $\epsilon>0$.
It is also relatively easy to see that it holds for the Poisson$+$Grid
p.p. $\Phi=\Phi_M+\Phi_G$.

The {\em SINR  at the receiver $X_j\in\Phi^0(n)$ with respect to transmitter
$X_i\in\Phi^1(n)$, at time $n$} is defined as 
\begin{equation}\label{eq:SINR}
\SINR_{i,j}(n)=\frac{F_{i,j}(n)/l(\norm{X_i-X_j})}{W_j(n)+I_{i,j}(n)}\,.
\end{equation}

\subsection{Space-Time SINR Graph}
\label{sec:oppTSSDiR}

Let
\begin{equation}\label{e.delta}
\delta_{i,j}(n)=
\begin{cases}\ind(\SINR_{i,j}\ge T)&\text{if}\   e_i(n)=1,e_j(n)=0, i\not=j,\\
1& \text{if}\ i=j,\\
0& \text{otherwise,}
\end{cases}
\end{equation}
where  $T>0$ is a threshold assumed to be 
some given constant throughout the paper.
We define the space-time SINR graph $\Gsinr$ as the {\em directed graph} 
with the set of vertexes $\Phi \times \bbZ$ and a 
directed edge from $(X_i,n)$ to $(X_j,n+1)$ if $\delta_{i,j}(n)=1$.

Let us stress an important convention in our terminology.
By network node, or point, we understand a point of $\Phi$. 
A (graph) vertex is an element of $\Phi\times\bbZ$;
i.e. it represents some network node at some time.
The existence of a graph edge is to be interpreted as the possibility 
of a successful communication between two network nodes (those 
involved in the edge) at time $n$. This can be rephrased as follows.
Suppose that at time $n$ the network node $X_i$ has a packet
(containing some information). Then the set of graph neighbors of the
vertex $(X_i,n)$ describes all the nodes that can decode this packet at time
$n+1$. Thus any path on the graph $\Gsinr$ represents some possible
route of the packet in space and time.

\section{Results}
\label{s.results}
In this section we present our results on $\Gsinr$.

\subsection{Existence of Paths}

All the results of this section are obtained under 
the general p.p. and fading assumptions of Section~\ref{sec:oppstomo},
under the assumption that the finiteness of the expression in~(\ref{e.mean_sn})
is granted.

Note first that $\Gsinr$ has no isolated nodes in the usual sense. Indeed, we 
have always $(X_i,n)$ connected to $(X_i,n+1)$.  
We will consider directed paths on $\Gsinr$ and call them paths for
short. Note that these paths are self-avoiding due to the fact that
there are no loops in the time dimension.

Denote by $\calH_{i}^{out,k}(n)$ 
the number of paths of length $k$ (i.e. with $k$ edges) 
{\em originating} from $(X_i,n)$. Similarly, denote by  $\calH_{i}^{in,k}(n)$ 
the number of such path {\em terminating} at $(X_i,n)$.
In particular  $\calH_{i}^{out}(n)=\calH_{i}^{out,1}(n)$
and $\calH_{i}^{in}(n)=\calH_{i}^{in,1}(n)$ are respectively, the out-
and in-degree of the node $(X_i,n)$.
\begin{lemma}
\label{lechopp:indeg-notime}
For a general p.p. $\Phi$ and a general fading model, the 
in-degree $\calH_i^{in}$ of any node of $\Gsinr$
is bounded from above by the constant $\xi=1/T +2$.
\end{lemma}
\begin{proof}
Assume there is an edge to node $(X_j,n)$ from
nodes $(X_{i_1},n-1),\ldots,(X_{i_k},n-1)$, for some $k>1$ and
$i_p\not=j$ ($p=1,\ldots,k$).
Then for all such $p$ 
$$ \frac{F_{{i_p},j}}{l(|X_{i_p}-X_j|)}\ge \frac{T}{1+T} \left(
\sum_{q=1}^k \frac{F_{{i_q},j}}{l(|X_{i_q}-X_j|)}\right)\, .$$
When summing up all these inequalities, one gets that $Tk\le 1+T$,
that is $k\leq 1/T+1$. Considering the edge from $(X_i,n-1)$ to $(X_i,n)$,
the in-degree of any node is bounded from above by $\xi=1/T +2$.
\qed
\end{proof}

Let $$h^{out,k} = \E^0 [\calH_{0}^{out,k}(n)] = \E^0 [\calH_{0}^{out,k}(0)]$$ 
and
$$h^{in,k} = \E^0 [\calH_{0}^{in,k}(n)] = \E^0 [\calH_{0}^{in,k}(0)]$$ 
be the expected numbers of paths of length $k$ originating or
terminating at the typical node, respectively.
In particular $h^{out}=h^{out,1}$ and 
$h^{in}=h^{in,1}$ are the mean out- and in-degree of
the typical node, respectively.

\begin{lemma}
\label{lem:choppnumpath}
For a general p.p. $\Phi$ and a general fading model 
\begin{equation}
h^{in,k}= h^{out,k}\,.
\end{equation}
\end{lemma}
\begin{proof}
We use the mass transport principle to get that 
$\E^0 [\calH_{0}^{out,k}(0)]= \E^0 [\calH_{0}^{in,k}(0)]$,
which implies the desired result. Indeed,
Campbell's formula and stationarity give
\begin{eqnarray*}
\lambda h^{out,k}&=&\lambda \int_{[0,1)^2}\E^0[\calH_{0}^{out,k}(0)]\md x\\
&=&\E\Bigl[\sum_{{X_i}\in \Phi\cap[0,1)^2} \calH_{i}^{out,k}(0)\Bigr]\\
&=&\sum_{v\in\bbZ}\E\Bigl[\sum_{X_i\in [0,1)^2}\sum_{X_j\in[0,1)^2+v}
\#\,\text{of paths from  $(X_i,0)$ to $(X_j,k)$}\Bigr]\\ 
&=&\sum_{v\in\bbZ}\E\Bigl[\sum_{X_i\in [0,1)^2-v}\sum_{X_j\in[0,1)^2}
\#\,\text{of paths from  $(X_i,0)$ to $(X_j,k)$}\Bigr]\\ 
&=&\lambda \int_{[0,1)^2}\E^0[\calH_{0}^{in,k}(k)]\md x =\lambda h^{in,k}\,,
\end{eqnarray*}
where $\#$ denotes the cardinality.
This completes the proof.
\qed
\end{proof}

Here are immediate consequences of the two above lemmas.
\begin{cor}\label{c.local-finiteness}
Under the assumptions of Lemma~\ref{lechopp:indeg-notime}:
\begin{itemize}
\item $\Gsinr$ is locally finite (both on in- and out-degrees of
  all nodes are $\Pro$-a.s. finite).
\item $\calH_{i}^{in,k}(n)\le\xi^k$ $\Pro$-a.s for all $i,n,k$.
\item $h^{in,k}= h^{out,k}\le \xi^k$ for all $k$.
\end{itemize}
\end{cor}

For all  $X_i,X_j\in\Phi$ and  $n\in\bbZ$, we denote by 
we will call {\em local delay} from $X_i$ to $X_j$ at time $n$ the quantity
$$\LD_{i,j}(n)=\inf\{k\ge n:\delta_{i,j}(k)=1\}$$
with the usual convention that $\inf\emptyset=\infty$. 
Note that $\LD_{i,j}(n)$ is the length (number of
edges) of the shortest path (with the smallest number of edges)
from $(X_i,n)$ to $\{X_j\}\times\bbZ$ among the paths contained in
the subgraph $\Gsinr\cap \{X_i,X_j\}\times\bbZ$ of $\Gsinr$, which is of the form 
\begin{eqnarray*}
& & ((X_i,n),(X_i,n+1)),\ldots,((X_i,n+\LD_{i,j}(n)-1),(X_i,n+\LD_{i,j}(n))),\\
& & ((X_i,n+\LD_{i,j}(n)),(X_j,n+\LD_{i,j}(n)+1)).
\end{eqnarray*}
Our next result gives a condition for the local delays to be a.s. finite.

\begin{lemma}\label{l.local-delay-finite}
Assume  a general p.p. $\Phi$ and a general fading model with $F$ having
unbounded support ($\Pro\{\,F>s\,\}>0$ for all $0<s<\infty$). Then,
given $\Phi$, all  local delays  $\LD_{i,j}(n)$ are $\Pro$-a.s. 
finite geometric random variables.
\end{lemma}
\begin{proof}
Due to our assumption on the independence of marks in successive time
slots, given $\Phi$, the variables $\{\delta_{i,j}(n):n\in\bbZ\}$ are (i.i.d.)
Bernoulli r.v. and thus $\LD_{i,j}(n)$ is geometric r.v. It remains to
show that $\Pro\{\, \delta_{i,j}(0)=1\,|\,\Phi\,\}:=\pi_{i,j}(\Phi)>0$ for
$\Pro$-almost all $\Phi$.
For this, note that
$$\pi_{i,j}(\Phi)=p(1-p)\Pro\Bigl\{\,F_{i,j}(0)\ge
l(\norm{X_j-X_i})\Bigl(W_j(0)+
I_{i,j}(0)\Bigr)\,\Bigr\}\,.$$ 
Under our general assumptions  (including finiteness of the
expression in~(\ref{e.mean_sn}))
$I_{i,j}(0)$ is a finite random variable $\Pro$-a.s. The result
follows from the assumption that $0<p<1$ and the fact that
$F_{i,j}(0)$ is independent  of $I_{i,j}(0), W_{i,j}(0)$ and has 
infinite support.
\qed
\end{proof}

The next result directly follows from Lemma~\ref{l.local-delay-finite}.

\begin{cor}
Under the assumptions of Lemma~\ref{l.local-delay-finite},
$\Gsinr$ is $\Pro$-a.s. {\em connected} in the
following {\em weak sense}:
for all $X_i,X_j\in\Phi$ and all $n\in \bbZ$,
there exists a path from $(X_i,n)$ to the set $\{(X_j,n+l):l\in \bbN\}$,
where $\bbN=\{1,2,\ldots\}$.
\end{cor}

We denote by 
$\LD_{i}(n)=\inf_{j\not=i}\LD_{i,j}$ the length of a shortest
directed path from $(X_i,n)$ to $(\{\Phi\setminus X_i\})\times\bbZ$.
We will call $\LD_{i}(n)$ the {\em exit delay} 
from $X_i$ at time $n$. Finally, we denote by $P_{i,j}(n)$ the
length of a shortest path of $\Gsinr$ from $(X_i,n)$ to
$\{X_j\}\times\bbZ$. We call  $P_{i,j}(n)$ the {\em delay}
from $X_i$ to $X_j$ at time $n$.
Obviously for $i\not=j$ we have
\begin{equation}\label{e.sandwich}
L_{i}(n)\le P_{i,j}(n)\le L_{i,j}(n)
\end{equation}
and thus it follows immediately from  Lemma~\ref{l.local-delay-finite} that
all the three collections of delays 
finite r.v.s $\Pro$-a.s. 

\subsection{Optimal Paths --- Poisson p.p. Case}
\label{ss.Poisson}
We have seen in the previous section that
under very general assumptions, all the delays 
are $\Pro$-a.s. finite random variables.
In this section we show that under some natural assumptions (such as Poisson
p.p. and exponential fading), the averaging over $\Phi$ may 
lead to {\em infinite} mean values. This averaging is expressed 
in terms of the expectation for the typical node under the Palm probability.
The proofs of the results stated in what follows are given in 
Section~\ref{ss.proofs.Poisson}.

Denote $\ld=\E^0[\LD_{0}(n)]=\E^0[\LD_{0}(0)]$.

\begin{prop}\label{p.multicast-infinite}
Assume $\Phi$ to be a Poisson p.p., $F$ to be exponential and the noise $W$ to be
bounded away from $0$: $\Pro\{\,W>w\,\}=1$ for some $w>0$. Let the path-loss
function be given by~(\ref{e.path-loss}). Then
$\Pro^0\{\,\LD_{0}(0)\ge q\,\}\ge 1/q$ for $q$ large enough. 
\end{prop}
\begin{cor}
Under the assumptions of Proposition~\ref{p.multicast-infinite}, we have:
\begin{itemize}
\item  The mean exit delay from the typical node is infinite;
$\ld=\infty$.
\item 
In any given subset of plane with positive Lebesgue measure,
at a given time, the expected number of points of $\Phi$ which have exit delays 
larger than $q$ decreases not faster than $1/q$ asymptotically for large $q$. 
\end{itemize}
\end{cor}

The fact that the mean exit delay from the typical point
is infinite ($\ell=\infty$) seems to be a consequence of the potential
existence of arbitrarily large ``voids'' (disks without points of $\Phi$)
around this point. Indeed, when conditioning on the existence of another
point in the configuration $\Phi$, one obtains finite mean local delays.
This will be shown in Proposition~\ref{p.point-to-point-finite} below.

Before stating it we need to formalise the notion of existence of {\em two given
points $X,Y\in\ir^2$ of $\Phi$}. For this, we consider $\Phi$ under
the two-fold Palm probability $\Pro^{X,Y}$. Since our results on
the matter bear only on the Poisson p.p. case, we can assume (by 
Slivnyak's Theorem) the following version of the Palm probability
of the Poisson p.p. $\Phi$:
\begin{equation}\label{e.Palm2}
\Pro^{X,Y}\{\,\Phi\in\cdot\,\}=
\Pro\Bigl\{\,\Phi\cup\{X,Y\}\in\cdot\,\Bigl\}\,.
\end{equation}
Moreover, under $\Pro^{X,Y}$, the marked Poisson p.p.  $\tilde\Phi$
is obtained by an independent marking of the points of $\Phi\cup\{X,Y\}$ 
according to the original distribution of marks.
Slightly abusing the notation, we denote by  $L_{X,Y}(n)$ the 
local delay from $X$ to $Y$ at time $n$ when considered under
$\Pro^{X,Y}$. Similar convention will be adopted in the notation of 
other types of delays under the Palm probabilities $\Pro^X$ or $\Pro^{X,Y}$.

\begin{prop}\label{p.point-to-point-finite}
Assume $\Phi$ to be a Poisson p.p., $F$ to be exponential and the noise $W$ to 
have a general distribution.
Then for all $X,Y\in\ir^2$, the mean local delay from $X$ to $Y$
is finite {\em given the existence of these two points in $\Phi$}. More precisely, 
\begin{equation}\label{e.point-to-point-finite}
\E^{X,Y} [L_{X,Y}(0)] < \infty\,.
\end{equation}
\end{prop}

The next result follows immediately from~(\ref{e.sandwich}). 
\begin{cor}\label{c.point-to-point-finite}
Under the assumptions of Proposition~\ref{p.point-to-point-finite},
\begin{equation}\label{e.point-to-point-finite-path}
\E^{X,Y} [L_{X}(0)]\le\E^{X,Y} [P_{X,Y}(0)] < \E^{X,Y} [L_{X,Y}(0)] <\infty\,.
\end{equation}
\end{cor}

The following result is our main ``negative'' result concerning $\Gsinr$ 
in the Poisson p.p. case:  

\begin{prop}\label{p.point-to-point}
Under the assumptions of Proposition~\ref{p.multicast-infinite},
we have 
\begin{equation}\label{e.poin-to-point-assymptotic}
\lim_{\norm{X-Y}\to \infty}
\frac{\E^{X,Y} [P_{X,Y}(0)]}{\norm{X-Y}} = \infty\,.
\end{equation}
\end{prop}

In other words, the {\em expected shortest delay necessary to send a packet between 
two given points of the Poisson p.p. grows faster than the Euclidean distance between
these two points}. 

\subsection{Filling in Poisson Voids} 
\label{ss.Kingman}
In this section we show that adding
an independent periodic pattern of points to the Poisson p.p. 
allows one to get a linear scaling of the shortest path delay 
with Euclidean distance. In order to prove the {\em existence and finiteness}
of the associated time constant, we adopt a slightly different approach
to the notion of paths on $\Gsinr$, which will allow us to exploit a subadditive
ergodic theorem. The proofs of the results stated in what follows are given
in Section~\ref{ss.proofs.Kingman}.

For $x\in \ir^2$, let $X(x)$ be the point of $\Phi$ which is
closest to $x$. The point $X(x)\in\Phi$ is a.s. well defined 
for all given $x\in\ir^2$ since $\Phi$ is  assumed simple and
stationary p.p. For all $x,y\in \ir^2$,  
define $P(x,y,n)=P_{X(x),X(y)}(n)$ to be the length of 
a shortest path of $\Gsinr$ from vertex $(X(s),n)$ to the set $\{(X(y),n+l),l\in\bbN\}$.
We will call $P(x,y,n)$ the {\em delay} from $x$ to $y$ at time $n$. 
For all triples of points $x,y,z\in\ir^2$, we have
\begin{equation} 
\label{eq:subadcor}
P(x,z,n) \le P(x,y,n) +P\Bigl(y,z,n+P(x,y,n)\Bigr)\,.
\end{equation}
Let
\begin{equation}
p(x,y,\Phi) =\E[P(x,y,0)\,|\,\Phi]\,.
\end{equation}
Using the strong Markov property, we get that, conditionally on $\Phi$,
the law of $P(y,z,n+P(x,y,n))$ is the same as that of $P(y,z,n)$. 
Then, the last relation and~(\ref{eq:subadcor}) give
\begin{eqnarray} 
\label{eqopp:subaddce}
p(x,z,\Phi) \le p(x,y,\Phi) +p(y,z,\Phi)\,.
\end{eqnarray}

We are now in a position to use the subadditive ergodic theorem
to show the existence of the time constant
$$\kappa_{\mathrm{d}}=\lim_{t\to\infty} 
\frac{p(0,t\mathrm{d},\Phi)}{t}\, ,$$
where  $\kappa_{\mathrm{d}}$ 
may depend on the {\em unit vector} $\mathrm{d}\in\ir^2$ representing
the direction in which the delay is measured.
Here is the main result of this section.
\begin{prop}\label{cor:opopprev}
Consider the Poisson$+$Grid p.p. defined in
Section~\ref{ss.Probab.assumptions} with exponential fading $F$ and with
the path-loss function be given by~(\ref{e.path-loss}).
Then, for all unit vectors $\mathrm{d}\in\ir^2$, 
the non-negative limit  $\kappa_{\mathrm{d}}$
exists and is $\Pro$-a.s. {\em finite}.
The convergence also holds in $L_1$.
\end{prop}

Notice that $\kappa_{\mathrm{d}}$ is not a constant.
Indeed, the superposition of the p.p.s $\Phi=\Phi_M$ and $\Phi_G$
is ergodic but not mixing due to the fact that 
$\Phi$ is a (stationary) grid. For $\mathrm{d}$ parallel to
say the horizontal axis of the grid $\Phi_G$, 
the limit $\kappa_{\mathrm{d}}$ will depend on the distance 
from the line $\{t\mathrm{d}:t\in\ir\}$ to the
nearest parallel (horizontal) line of the grid $\Phi_G$, i.e. on the shift 
$U_G$ of the grid. Here is a more precise formulation of the result.
\begin{prop}\label{p.kappa_constant}
Under the assumptions of Proposition~\ref{cor:opopprev},
the limit $\kappa_{\mathrm{d}}=\kappa_{\mathrm{d}}(U_G)$ is measurable
w.r.t. the shift $U_G$ of the grid p.p. $\Phi_G$ and {\em does not} depend
on the Poisson component $\Phi_M$ of the p.p. $\Phi$.
Moreover, the set of vectors $\mathrm{d}$ in the unit sphere for which
$\kappa_{\mathrm{d}}(U_G)$ is not $\Pro$-a.s. a constant
is at most countable.
\end{prop}

The last result on this case is:
\begin{prop}
\label{lemchopplemdiff}
Under the assumptions of Proposition~\ref{cor:opopprev},
suppose that $W$ is constant and strictly positive.
Then $\E[\kappa_{\mathrm{d}}]>0$. 
\end{prop}

Finally let us remark that the 
method used in this section cannot be
used in the case of the Poisson p.p. 
(without the addition of the grid point process).
The main problem is the lack of integrability 
of $p(x,y,\Phi)$ as stated in the following result.
Note however, that this does {\em not} imply immediately  that 
$\kappa_{\mathrm{d}}=\infty$. 

\begin{cor}\label{c.pPoissonInfty}
Under the assumptions of Proposition~\ref{p.multicast-infinite} 
$\E[p(x,y,\Phi)]=\infty$
for all $x$ and $y$ in $\ir^2$.
\end{cor}

\section{Proofs}
\label{s.proofs}
Consider the shortest path 
from $(X_i,n)$ to $(\Phi\setminus\{X_i\})\times\bbZ$. 
Let $\Trial_i(n)$ be the 
number of edges $(X_i,k),(X_i,k+1)$ in this path
such that $e_i(k)=1$.
These variables are the {\em number of trials} before the first exit 
form $X_i$ at time $n$.
Obviously 
\begin{equation}\label{e.trials_bound}
\Trial_i(n)\le\LD_i(n)\,.
\end{equation}

We will also consider 
an auxiliary graph $\widehat\Gsinr$, called the {\em (space-time)
Signal to Noise Ratio (SNR) graph}, defined 
exactly in the same manner as the SINR graph $\Gsinr$ except that
the variables $\SINR_{i,j}(n)$ defined in~(\ref{eq:SINR}) are replaced 
by the variables 
\begin{equation}\label{eq:SINR}
\SNR_{i,j}(n)=\frac{F_{i,j}(n)/l(\norm{X_i-X_j})}{W_j(n)}\,.
\end{equation}
Note that this modification consists in suppressing the interference
term $I_{i,j}(n)$ in the SINR condition in~(\ref{e.delta}).
The  edges of $\Gsinr$ form a subset of the edges of $\widehat\Gsinr$
(both graph share the  same vertexes), which will be denoted by
\begin{equation}\label{e.inclusion}
\Gsinr\subset\widehat\Gsinr\,.
\end{equation}
In what follows we will denote the delays, local
delays, exit delays and numbers of trials related to $\widehat\Gsinr$ 
by $\widehat P_{i,j}(n), \widehat\LD_{i,j}(n), \widehat\LD_{i}(n)$ and
$\widehat\Trial_{i}(n)$, respectively. The inclusion
$\Gsinr\subset\widehat\Gsinr$ implies immediately that
$\widehat P_{i,j}(n)\le P_{i,j}(n)$ and the same inequalities hold 
for the three other families of variables mentioned above.

\subsection{Proofs of Results of Section~\ref{ss.Poisson}}
\label{ss.proofs.Poisson}

\begin{proof}({\em of Proposition~\ref{p.multicast-infinite}})
The inclusion~(\ref{e.inclusion}) and the
inequality~(\ref{e.trials_bound}) yield
$$\widehat\Trial_{i}(n)\le \Trial_{i}(n)\le \LD_{i}(n)\,,$$
which holds for all $i,n$. The results follow from the above inequalities
and the next lemma.
\qed
\end{proof}

\begin{lemma}\label{l.multicast-infinite}
Under the assumptions of Proposition~\ref{p.multicast-infinite},
$\Pro^0\{\,\widehat\Trial_{0}(0)\ge q\,\}\ge 1/q$ for $q$ large enough. 
\end{lemma}
\begin{proof}
Under $\Pro^0$, denote by $\tau_k$ the $k\,$th time slot in $\{0,1,\ldots\}$, 
such that \hbox{$e_0(k)=1$}. 
For all $q\ge 0$ we have
\begin{eqnarray*}
\Pro^0\{\,\widehat\Trial_0(0)>q\,|\,\Phi\,\}&=&
\Pro^0\Bigl\{\,\forall_{0\le k\le q}
\forall_{0\not=X_i\in\Phi}\;\delta_{0,i}(\tau_k)=0
\,\Bigl|\,\Phi\,\Bigr\}\\
&=&\Pro^0\Bigl\{\,\forall_{0\le k\le q}
\forall_{0\not=X_i\in\Phi}\;e_i(\tau_k)=1 \text{\;or\;}\SNR_{0,i}(\tau_k)<T
\,\Big|\,\Phi\,\Bigr\}
\end{eqnarray*}
and by  the conditional independence of marks given $\Phi$ 
\begin{eqnarray*}
\Pro^0\{\,\widehat\Trial_0(0)>q\,|\,\Phi\,\}&=&
\prod_{0\not=X_i\in\Phi}\Bigl(p+(1-p)\Pro\{F<Tl(\norm{X_i})W\}\Bigr)^q\\
&=&\exp\Bigl\{q\sum_{0\not=X_i\in\Phi}
\log\Bigl(p+(1-p)(1-e^{-\mu Tl(\norm{X_i})W})\Bigr)\Bigr\}\,,
\end{eqnarray*}
where $F,W$ are independent generic random variables representing
fading and thermal noise, independent of $\Phi$, $F$ is exponential
with mean $1/\mu$, 
Using the  Laplace functional formula for $\Phi$ and the assumption
that $W>w$ a.s. we have
\begin{eqnarray}
\Pro^0\{\,\widehat\Trial_{0}(0)\ge q\,\}
&\ge&\nonumber
\exp\left(-2\pi \lambda 
\int_{v>0} \left(1-\left(1-(1-p)e^{-w\mu l(v)T} \right)^q
\right) v \md v\right)\\
&=&
\exp\left(-\pi \lambda 
\int_{v>0} \left(1-\left(1-f(v)\right)^q
\right)\md v\right)\,,\label{e.exp_ge1overq}
\end{eqnarray}
where 
$$f(v):=
(1-p) \exp(- K v^{\beta/2}) \quad \text{and}\quad K=w \mu T A^\beta\,.
$$ 
In what follows we will show that the expression in~(\ref{e.exp_ge1overq})
is not smaller  than $1/q$
for $q$ large enough.
To this regard 
denote by $v_q$ the unique solution of 
$f(v)=\frac 1 q$.
We have
$$ v_q= \frac 1 {A^2 \left(\mu T w\right) ^{2/\beta}} (\log(q(1-p)))^{2/\beta}.$$
It is clear that $f(v)$ tends to~0 when $v$ tends to infinity
and that  $v_q$ tends to infinity as $q$ tends to infinity.
Therefore, 
there exists a constant $Q=Q(\mu,w,A,T)<\infty$ such that
for all $q\ge Q$ and for all $v\ge v_q$,
$$ (1-f(v))\ge \exp(-f(v)).$$
Hence, for all $q\ge Q$,
\begin{eqnarray*}
\int_{v>0} \left(1-\left(1-f(v) \right)^q \right)  \md v 
&\le &
v_q + \int_{v_q}^\infty
 \left(1-\left(1-f(v) \right)^q \right)  \md v \\
&\le &
v_q + \int_{v=v_q}^\infty
\left(1-\exp(-q f(v) \right)\, \md v\\
&\leq &  
v_q + \int_{v_q}^\infty  q f(v)\, \md v\\
&= &  v_q + \int_{u=0}^\infty q f(u+v_q)\, \md u.
\end{eqnarray*}
The third inequality follows from the
fact that $1-\exp(-x)\le x$.
Using now the fact that $(u+v_q)^{\beta/2}\ge u+ v_q^{\beta/2}$
(for $q$ large enough, say again $q\ge Q$) 
we get that
\begin{eqnarray*}
\int_{u=0}^\infty
q f(u+v_q)\, \md u &  = &
\int_{u=0}^\infty
q (1-p) \exp(-K (u+v_q)^{\beta/2}))\, \md u\\
& \le &
\int_{u=0}^\infty
q (1-p) \exp(-K u -K v_q^{\beta/2})\, \md u
=  \frac 1 K\, ,
\end{eqnarray*}
since $(1-p)\exp(-K v_q^{\beta/2})=1/q$.
Hence for $q\ge Q$
$$
\int_{v>0} \left(1-\left(1-f(v)\right)^q \right)\,  \md v 
\le v_q + \frac \alpha K.$$
Also it is not difficult to see that $\beta>2$ implies 
\begin{equation}
\label{v_m_lim}
v_q\le\frac{\log q}{\pi\lambda}-\frac{1}{K}
\end{equation}
for $q$ large enough. This implies for $q$ large enough,  say again $q\ge Q$, 
\begin{equation}\label{e.1overq}
\exp\left(-\pi \lambda 
\int_{v>0} \left(1-\left(1-f(v)\right)^q
\right)\md v\right)\ge
\exp\Bigl(-\pi\lambda (v_q+1/K)\Bigr)\ge\frac 1 q\,,
\end{equation}
which completes the proof.
\qed
\end{proof}

\begin{proof}({\em of Proposition~\ref{p.point-to-point-finite}}).
 Assume without loss of generality 
$Y=0$ and $\norm{X}=r$.
Under $\Pro$, consider the p.p. $\Phi\cup\{X,0\}$
and its independent marking.
Given $\Phi$, the r.v. $\LD_{X,0}(0)$ associated with the independently marked
p.p.  $\Phi\cup\{X,0\}$ has a geometric distribution with parameter 
$$\pi_{X,0}(\Phi)=p(1-p)\Pr\Bigl\{\,F\ge
l(r)(W+I)\Bigr)\,\Bigr\}\,,$$
where $F,W,I$ are independent r.v.s, $F,W$ are generic fading and noise
variables and
$I=\sum_{X_i\in\Phi}e_i(0)F_{i,0}(0)/l(\norm{X_i})$.
Using the exponential distribution of $F$ and the independence, we obtain
$$\pi_{X,0}(\Phi)=
\E[e^{-\mu l(r)TW}]\;\E[e^{-\mu l(r)TI}\,|\,\Phi]\,.$$
The mean of the geometric r.v. is known to be 
$\E^{X,0}[\LD_{X,0}(0)\,|\,\Phi]=1/\pi_{X,0}(\Phi)$. By unconditioning 
with respect to $\Phi$, one obtains
$$\E^{X,0}[\LD_{X,0}(0)]=
\frac1{\calL_W(\mu l(r)T)}
\E\Bigl[\frac1{\E[e^{-\mu l(r)TI}\,|\,\Phi]}\Bigr]\,.$$
The first factor in the above expression is obviously finite.
In what follows we will evaluate the second one.

By the conditional independence of marks and denoting by 
$\calL_{eF}(\cdot)$ is the Laplace transform of $eF$,
where $e,F$ are independent generic variables for $e_i(0)$ and
$F_{i,0}(0)$ we have 
\begin{eqnarray*}
\Bigl(\E\Bigl[e^{-\mu l(r)TI}\,|\,\Phi]\Bigr]\Bigr)^{-1}&=&
\left(\E\left[\exp\Bigl(-\mu l(r)T\sum_{X_i\in\Phi}e_i(0)F_{i,0}(n)/l(\norm{X_i})\Bigr)\,\Big|\,\Phi\right]\right)^{-1}\\
&=&\exp\left(\sum_{X_i\in\Phi}
\log\calL_{eF}\Bigl(\mu T l(r)/l(\norm{X_i})\Bigr)\right)\, .
\end{eqnarray*}
Note that 
$\calL_{eF}(\xi)=1-p+p\calL_F(\xi)=1-p+p\mu/(\mu+\xi)$. 
Using this and the  Laplace functional formula for $\Phi$,
(cf.~\cite[Eq.~9.4.17]{DVJII2008}) we obtain
\begin{eqnarray*}
\E\Bigl[\frac1{\E[e^{-\mu l(r)TI}\,|\,\Phi]}\Bigr]&=&
\exp\biggl\{2\pi p\lambda\int_0^\infty
\frac{vTl(r)}{l(v)+(1-p)Tl(r)}\,\md v\biggr\}\,.
\end{eqnarray*}
(cf.~(\ref{e.mean_sn})
Using now the fact that for the Poisson p.p.,
$\breve M_{[2]}(\myd x)=\lambda \myd x$),
it is now easy to see that for any path-loss function 
satisfying $\int_\epsilon^\infty  v/l(v)\md v<\infty$,
the integral in the exponent of the last expression is finite. This completes the proof.
\qed
\end{proof}

\begin{proof}({\em of Proposition~\ref{p.point-to-point}}).
Using the the
inclusion~(\ref{e.inclusion}), inequality~(\ref{e.trials_bound})
and the left-hand side of~(\ref{e.sandwich}) and we have
$$\widehat\Trial_{i}(n)\le \Trial_{i}(n)\le \LD_{i}(n)\le P_{i,j}(n)\,.$$
Thus, it is enough to show
$$
\lim_{\norm{X-Y}\to \infty}
\frac{\E^{X,Y} [\widehat\Trial_{X}(0)]}{\norm{X-Y}} = \infty\,.
$$
Without loss of generality assume $X=0$ and $\norm{Y}=r$.
Using the same arguments as in the proof of 
Lemma~\ref{l.multicast-infinite} and the
representation~(\ref{e.Palm2})
of the Palm probability  with respect to Poisson p.p., we obtain 
\begin{eqnarray*}
\lefteqn{\Pro^{0,Y}\{\,\widehat\Trial_0(0)>q\,|\,\Phi\,\}}\\
&\ge&
\prod_{0,Y\not=X_i\in\Phi}\Bigl(p+(1-p)\Pro\{F<Tl(\norm{X_i})W\}\Bigr)^q
\;\Bigl(p+(1-p)\Pro\{F<Tl(\norm{Y})W\}\Bigr)^q
\\
&\ge&\exp\left(-\pi \lambda 
\int_{v>0} \left(1-\left(1-f(v)\right)^q
\right)\md v\right)
\; \alpha(r)^q\,,
\end{eqnarray*}
where $\alpha(r)= 1-(1-p) e^{-w\mu A^\alpha T r^\beta}$.
Using~(\ref{e.1overq}), which holds for large $q$, more precisely
$q>Q=Q(\mu,w,A,T)$, we obtain
\begin{equation*}
\frac{\E^{0,Y}[\widehat\Trial_{0}(0)]}{r} 
\ge  \frac 1{r}  \sum_{q>Q} \frac {\alpha(r)^q} q\,.
\end{equation*}
It is now easy to see that 
$$ \lim_{r\to \infty} \frac 1 r \sum_{q>Q} \frac {\alpha(r)^q} q =\infty.$$
\qed
\end{proof}

\subsection{Proofs of Results of Section~\ref{ss.Kingman}}
\label{ss.proofs.Kingman}
Denote by $B_x(R)$ the ball centered at $x\in\ir^2$ of radius $R$.
Similarly as for the delays, we extend the definition of the local
delays to arbitrary pairs of points $x,y\in\ir^2$ by taking
$\LD(x,y,n)=\LD_{X(x),X(y)}(n)$.
We first establish the following technical result:
\begin{lemma}\label{l.LDxyMGs_bound}
Under the assumptions of Proposition~\ref{cor:opopprev}
let $X_i,X_j\in \Phi\cap B_0(R)$ for some $R>0$, where $\Phi=\Phi_M+\Phi_{G_s}$.
Then the conditional expectation of the local delay $\LD_{i,j}(0)$ 
given $\Phi$ satisfies 
\begin{eqnarray}\label{e.LDxyMGs_bound}
\lefteqn{\E[\LD_{i,j}(0)\,|\,\Phi]}\nonumber\\
&=&
\frac1{p(1-p)\calL_W(T\mu A^\beta \norm{X_i-X_j}^\beta)}
\exp\Bigl\{-\hspace{-.2cm}\sum_{\Phi\owns X_k, k\not=i,k}\log \calL_{eF'}\Bigl(
\frac{T\norm{X_i-X_j}^\beta}{|X_j-X_k|^\beta}\Bigr)\Bigr\}\\
&\le&\frac1{p(1-p)\calL_W(T\mu (A2R)^\beta)}\nonumber\\[1ex]
&&\times e^{-49\log(1-p)+(2R)^\beta pTC(s,\beta)}
\hspace{8.9cm}(a)\nonumber\\[1ex]
&&\times e^{-\Phi_M(B_0(2R))\log(1-p)}
\hspace{9.75cm}(b)\nonumber\\[1ex]
&&\times\exp\Bigl\{-\sum_{X_k\in\Phi_M,|X_k|>2R}
\log\Bigl(1-p+\frac{p(|X_k|-R)^{\beta}}%
{(|X_k|-R)^{\beta}+T(2R)^{\beta}}\Bigr)
\Bigr\}\,,
\hspace{3.55cm}(c)\nonumber
\end{eqnarray}
where $C(s,\beta)<\infty$ is some constant (which depends on $s$ and
$\beta$ but not on $\Phi$), $F'$ is an exponential random variable of mean 1
and $\calL_{eF'}(\cdot)$ is the Laplace transform of $eF'$.
\end{lemma}
\begin{proof}
We first prove the equality in~(\ref{e.LDxyMGs_bound}). When using the independence
assumptions, we have
\begin{eqnarray*}
\lefteqn{\Pro \left\{\,\LD_{i,j}(0)>m\mid\Phi\,\right\}}\\
& = &
\Pro \left\{\,\forall_{n=1}^m
  \left(e_j(n)=1\;\text{or}\;\right.\right.\\
&&\hspace{4em}\left.\left.
 e_j(n)=0\;\text{and}\;e_i(n) F_{i,j}(n) \le T l(\norm{X_i-X_j})(W_j(n)+
 I_{i,j}(n))\right)\,\right\}\\
& = & \prod_{n=1}^m \Biggl( p+ (1-p)\biggl(1-p +p  \Bigl(1-
\calL_W(T\mu A^\beta |x-y|^\beta)\\[-2ex]
&&\hspace{15em}\times
\prod_{\Phi\owns X_k, k\not=i,j} \calL_{eF'}\Bigl(
\frac{T\norm{X_i-X_j}^\beta}%
{\norm{X_j-X_k}^\beta}\Bigr)\Bigr)\biggr)\Biggr)\, .
\end{eqnarray*}
The result then follows from the evaluation of
$$\E \left[\LD_{i,j}(0)\,\mid\,\Phi\right]  = \sum_{m=0}^\infty
 \Pro \left\{\,\LD_{i,j}(0)>m\,\mid\,\Phi\right]\,.$$

The bound $|X_i-X_j|\le 2R$ used in the Laplace transform
of $W$ leads to the first factor of the upper bound. 
We now factorize the exponential function
in~(\ref{e.LDxyMGs_bound}) as the product of three exponential functions
\begin{eqnarray*}
\alpha&:=& \exp\Bigl\{-\sum_{\Phi_{G_s}\owns X_k, k\not=i,j}\Bigr\},\\
 \beta&:=& 
\exp\Bigl\{-\sum_{\Phi_M\owns X_k, k\not=i,j |X_k|\le 2R}\Bigr\},\\
\gamma&:=& 
\exp\Bigl\{-\sum_{\Phi_M\owns X_k, |X_k|>2R}\Bigr\}.
\end{eqnarray*}
Next we prove that the last three exponentials
are upper-bounded by (a), (b) and (c) in (\ref{e.LDxyMGs_bound}), respectively.
\begin{itemize}
\item[(a)] We use $\norm{X_i-X_j}\le 2R$ and Jensen's inequality to get
\begin{eqnarray*}
\log\calL_{eF'}\Bigl(\frac{T\norm{X_i-X_j}^\beta}{|X_j-X_k|^\beta}\Bigr)
&\ge& 
\log\calL_{eF'}\Bigl(\frac{T(2R)^\beta}{|X_j-X_k|^\beta}\Bigr)\\
&\ge& \frac{-T(2R)^\beta\E[eF']}{|X_j-X_k|^\beta}\\
&=&-pT(2R)^\beta|X_j-X_k|^{-\beta}\,.
\end{eqnarray*}
We now prove that
$$\sum_{\Phi_{G_s}\owns X_k: |X_j-X_k|>3\sqrt2 s}|X_j-X_k|^{-\beta}\le C(s,\beta),$$
for some constant $C(s,\beta)$.
This follows from an upper-bounding of the value of
$|X_j-X_k|^{-\beta}$ by the value of the integral $1/s^2\int
(|X_j-x|-\sqrt2s)^{-\beta}\,\md x$ over the square with corner points
$X_k$, $X_k+(s,0)$, $X_k+(0,s)$ and $X_k+(s,s)$. In this way we obtain
\begin{eqnarray*}
\sum_{\Phi_{G_s}\owns X_k: |X_j-X_k|>3\sqrt2 s}|X_j-X_k|^{-\beta}&\le&
\frac1{s^2}\int_{|x-X_j|>2\sqrt2 s}^\infty (|X_j-x|-\sqrt2 s)^{-\beta}\,\md x\\
&=&\frac{2\pi}{s^2}\int_{\sqrt2 s}^\infty \frac{t+\sqrt2 s}{t^\beta}\,\md t=:C(s,\beta)<\infty\,. 
\end{eqnarray*}
Combining this and what precedes, we get that
$$
\exp\left\{
-\sum_{X_k\in\Phi_{G_s}, |X_j-X_k|>2\sqrt s}
\log\calL_{eF'}\Bigl(\frac{T\norm{X_j-X_i}^\beta}{|X_j-X_k|^\beta}\Bigr)
\right\}
\le \exp(T(2R)^\beta C(s,\beta)).
$$
We also have
$$\log\calL_{eF'}\Bigl(\frac{T(2R)^\beta}{|y-X_i|^\beta}\Bigr)\ge 
\log\calL_{eF'}(\infty)=\log(1-p)\, ,$$ 
for all
$X_k\in \Phi_{G_s}$ and in particular for $ |X_j-X_k|\le 3\sqrt2 s$. Hence we obtain
$$
\exp\{-\sum_{X_k\in\Phi_{G_s}}(\dots)\}\le
e^{-49\log(1-p)+T(2R)^\beta C(s,\beta)}\, ,$$
where 49 upper-bounds the number of points $X_k\in\Phi_{G_s}$ such that $|X_j-X_k|\le 3\sqrt2 s$.
\item[(b)]
Using the bound 
$|X_j-X_i|\le 2R$ and the inequality
$\log\calL_{eF'}(\xi)\ge\log\calL_{eF'}(\infty)=\log(1-p)$, 
we obtain
$$\exp\{-\sum_{\Phi_M\owns X_k, k\not=i,j, |X_i|\le 2R}(\dots)\}
\le e^{-\Phi_M(B_0(2R))\log(1-p)}\,.$$

\item[(c)]
Using the bounds $|X_j-X_i|\le 2R$ and $|X_j-X_k|\ge |X_k|-R$ (the
latter follows from  
the triangle inequality) and the expression 
$\calL_{eF'}(\xi)=1-p+\frac{p}{1+\xi}$, we 
obtain 
\begin{eqnarray*}
\lefteqn{\exp\Bigl\{-\sum_{\Phi_M\owns X_k, |X_k|>2R}(\dots)\Bigr\}}\\
&\le& \exp\Bigl\{-\sum_{X_k\in\Phi_M,|X_k|>2R}
\log\Bigl(1-p+\frac{p(|X_k|-R)^{\beta}}%
{(|X_k|-R)^{\beta}+T(2R)^{\beta}}\Bigr)\Bigr\}\,.
\end{eqnarray*}
This completes the proof.
\end{itemize}
\qed
\end{proof}

We can now prove the following auxiliary result.
\begin{lemma}\label{p:integrability-sup}
Under the assumptions of Proposition~\ref{cor:opopprev}
for all points $x,y$ of $\ir^2$,
$$\E\left[\sup_{x_1,y_1\in[x,y]}p(x_1,y_1,\Phi)\right]< \infty\,,$$
where the supremum is taken over $x_1,y_1$ belonging to the interval 
$[x,y]\subset\ir^2$.
\end{lemma}

\begin{proof}
Without loss of generality, we assume that $(x+y)/2=0$ is the origin
of the plane. 
Let $B=B_0(R)$ be the ball centered at $0$ and of radius $R$ such that 
no modification of the points in the complement of $B$
modifies $X(z)$ for any $z\in[x,y]$
(recall that $X(z)$ is the point of $\Phi$ which is the closest from $z$).
Since $\Phi=\Phi_M+\Phi_{G_s}$, with $\Phi_{G_S}$ the square lattice
p.p. with intensity $1/s^2$, 
it suffices to take $R=|u-v|/2+\sqrt{2} s$.
Let $B'=B_0(2R)$.
By the above choice of $B$ and the inequality~(\ref{e.sandwich})
we have for all $x_1,y_1\in[x,y]$
$$P(x_1,y_1,0)\le
\sum_{X_i,X_j\in\Phi\cap B}\LD_{i,j}(0)\,
$$
and consequently
$$\sup_{x_1,y_1\in[x,y]}p(x_1,y_1,\Phi)\le
\sum_{X_i,X_j\in\Phi\cap B}\E[\LD_{i,j}(0)\,|\,\Phi]\,.$$
Using the result of Lemma~\ref{e.LDxyMGs_bound} we obtain
\begin{eqnarray*}
\lefteqn{\sup_{x_1,y_1\in[x,y]}\overline{|p^*(x_1,y_1,\Phi)|}}\\
&\le&
\frac{ e^{-49\log(1-p)+(2R)^\beta pTC(s,\beta)}}%
{p(1-p)\calL_W(T\mu A(2R)^\beta)}\nonumber\\[1ex]
&&\times \exp\Bigl\{-\sum_{X_k\in\Phi_M,|X_k|>2R}
\log\Bigl(1-p+\frac{p(|X_k|-R)^{\beta}}%
{(|X_k|-R)^{\beta}+T(2R)^{\beta}}\Bigr)\Bigr\}\\
&&\times \Bigl(\Phi_M(B)+ \pi (R+\sqrt{2} s)^2/s^2\Bigr)
 e^{-\Phi_M(B')\log(1-p)}\,,
\end{eqnarray*} where $\pi (R+\sqrt{2} s)^2/s^2$ is an upper bound of the number of
points of $\Phi_{G_s}$ in~$B$.
The first factor in the above upper bound is deterministic.
The two other factors are random and independent due to the independence
property of the Poisson p.p. The finiteness of the expectation 
of the last expression follows from the  finiteness of the exponential moments
(of any order) of the Poisson random variable $\Phi_M(B')$.
For the expectation of the second (exponential) factor, we use 
the known form of the Laplace transform of the Poisson SN
to obtain the following expression
$$\E\Bigl[\exp\Bigl\{-\sum (\dots)\Bigr\}\Bigr]
=\exp\biggl\{2\pi p\lambda_M\int_{R}^\infty
\frac{T(2R)^\beta}{v^\beta+(1-p)T(2R)^\beta}(v+R)\,\md v\Bigr\}<\infty\,.
$$
\qed
\end{proof}

\begin{proof}({\em of Proposition~\ref{cor:opopprev}}
The existence and finiteness of the limit $\kappa_{\mathrm{d}}$
follows  from the subadditivity~(\ref{eqopp:subaddce}) and
Lemma~\ref{p:integrability-sup}
by  the continuous-parameter sub-additive ergodic
theorem (see~\cite[Theorem~4]{Kingman73}). 
\qed
\end{proof}

\begin{proof}({\em of Proposition~\ref{p.kappa_constant}})
First, we prove the  second statement; i.e., that
$\kappa_{\mathrm{d}}$ is constant for all
$\mathrm{d}$ in the unit sphere off some countable subset. Note  that 
the point process  $\Phi$ is ergodic as the independent 
superposition of  mixing Poisson p.p. $\Phi_M$ and ergodic 
grid process $\Phi_G$. This can be easily proved using e.g. the
respective characterisations of above  properties by means of Laplace
transforms of p.p. (see~\cite[Prop.~12.3.VI]{DVJII2008}).
From the ergodicity of $\Phi$ we {\em cannot} conclude 
the  desired property for any vector $\mathrm{d}$ since  the limit   
$\kappa_{\mathrm{d}}=\kappa_{\mathrm{d}(\Phi)}$ is not necessarily 
invariant with respect to translations of $\Phi$ by {\em any} vector
$x\in\ir^2$ but only $x=\alpha \mathrm{d}$ for any scalar
$\alpha\in\ir$.  The announced result follows from~\cite[Th.~1]{PughShub71}.

For the first statement, consider a product space with on which two independent
p.p.s $(\Phi_M,\Phi_G)$ are defined. Fix some vector $\mathrm{d}$
and define the operator
$T=T_1\times T_2$  on this
product space as the product of two operators, which 
correspond to the shift in the direction
$\mathrm{d}$, of  $\Phi_M$ and $\Phi_G$ respectively.
The $\sigma$-field invariant with respect to $T$ 
is the product of the respective $\sigma$-fields invariant with respect
to $T_1$ and $T_2$. The latter is trivial since 
$\Phi_M$ is mixing (as a Poisson p.p.).
Consequently every function of $(\Phi_M,\Phi_G)$ that is
invariant with respect to the shift in the direction $\mathrm{d}$ of
its first argument ($\Phi_M$) is a.s. constant. This concludes the
proof that $\kappa_{\mathrm{d}}$ is constant in $\Phi_M$
and thus depends only on $U_G$.
\end{proof}

\begin{proof}({\em of Proposition~\ref{lemchopplemdiff}})
\noindent
For a given path 
$\sigma=\{(X_0,n_0),(X_1,n_0+1),\ldots,\allowbreak (X_k,n_0+k)\}$ on $\Gsinr$
denote by $\norm\sigma=\sum_{i=1}^k\norm{X_i-X_{i-1}}$  the Euclidean
length of the projection of $\sigma$ on $\ir^2$; let us call it
Euclidean length of $\sigma$ for short and recall that 
the (graph) length of $\sigma$ is equal to $k$.
For fixed $\epsilon>0$ and all $n\ge 1$ 
denote by {\em $\Pi(n)=\Pi_\epsilon(n)$ the event 
that there exists a  path on $\Gsinr$ starting at $(X(0),0)$ 
that has (graph) length $n$  and  Euclidean length
larger than $n/\epsilon$}.

Assume  $\E[\kappa_{\mathrm{d}}]=0$. We show first that this implies 
that for any $\epsilon>0$, $\Pro^0$-a.s. the event $\Pi_\epsilon(n)$
holds for infinitely many $n$ 
\begin{equation}\label{e.Pi-infoften}
\Pro^0\Bigl\{\,\bigcap_{n\ge 1}\bigcup_{k\ge n}\Pi_\epsilon(k)\,\Bigl\}=1\,.
\end{equation}
Indeed, $\E[\kappa_{\mathrm{d}}]=0$ implies 
$\kappa_{\mathrm{d}}=0$ $\Pro$-a.s. and by 
Palm-Matthes definition of the Palm probability 
$\Pro^0$-a.s. as well.
This  means 
that $\E^0[P(0,t\mathrm{d},0)\,|\,\Phi]/t\to 0$, when $t\to\infty$,
which implies that 
\begin{equation}\label{e.lim0}
\lim_kP(0,t_k\mathrm{d},0)/t_k\to 0
\end{equation}
$\Pro^0$-a.s. for some subsequence $\{t_k:k\ge 1\}$,  with $\lim_kt_k=\infty$.
Recall that $P(0,t_k\mathrm{d},0)$ is the length of a shortest path 
from $(X(0),0)$ (with $X(0)=0$ under $\Pro^0$)
to $\{(X(t_k\mathrm{d}),n):n\ge 0\}$.
Denote one of such shortest paths by $\sigma_k$.
By the triangle inequality its Euclidean length satisfies 
\begin{equation}\label{e.euclid-norm}
\norm{\sigma_k}\ge \norm{0-X(t_k\mathrm{d})}\ge t_k-\sqrt2 s\,.
\end{equation}
From (\ref{e.lim0})
and~(\ref{e.euclid-norm}) 
one concludes that for any $\epsilon>0$ and $k$ large enough
the length of the path $\sigma_k$ is smaller than $\epsilon$ time 
its Euclidean length $\norm{\sigma_k}$.
Now, (\ref{e.Pi-infoften}) follows from the fact that 
the length of the path $\sigma_k$ tends to infinity with
$k$, which is a consequence of $t_k\to\infty$ and 
the local finiteness of the graph $\Gsinr$ (cf
Corollary~\ref{c.local-finiteness}).

We conclude the proof by showing that 
for $\epsilon$ small enough, 
\begin{equation}\label{e.BC}
\sum_n \Pro^0\{\Pi_\epsilon(n)\}<\infty\,,
\end{equation}
which  by the Borel--Cantelli lemma
implies that $\Pi(n)$
holds $\Pro^0$-a.s. only for a finite number of integers $n$ and thus
contradicts to~(\ref{e.Pi-infoften}).
To this regard assume constant $W=w>0$
and let ${\cal P}_w^n$ denote the set of paths $\sigma$ of 
$\Gsinr$ of length  $n$,  originating from $(X(0)=0,0)$. 
Denote also by ${\cal P}_0^n$ the analogous set of paths on the
graph constructed under assumption $W=0$.
Note that by monotonicity, 
\begin{equation}
\label{eq:choppinclu}
{\cal P}^n_w\subset {\cal P}^n_0.
\end{equation}
By the definition 
\begin{equation}
\label{eq:choppdefevent}
\Pro^0\{\Pi_\epsilon(n)\,|\,\Phi\,\} = \Pro^0\Bigl(\bigcup_{\sigma}
\Bigl\{\,\sigma \in {\cal P}_w^n\;\text{and}\;
\norm{\sigma}\ge  n/\epsilon\,\Bigr\}\,\Big|\,\Phi\Bigr)\,,
\end{equation}
where the sum bears on all possible $n$-tuples 
$\sigma=((X_{j_1},1),\ldots,(X_{j_n},n))$, with $X_{j_i}\in\Phi$.
From this we have
\begin{eqnarray}
\label{eq:choppdefevent2}
\lefteqn{\Pro^0(\Pi_\epsilon(n)\,|\,\Phi)}\\
 & \leq &
\sum_{\sigma}
\Pro^0\Bigl\{\,\sigma \in {\cal P}^n_W,
|\sigma |\geq n/\epsilon\,\Big|\,\Phi\,\Bigr\}\nonumber \\
& = &\nonumber
\sum_{\sigma}
\Pro^0\Bigl\{\,\sigma \in {\cal P}^n_W,
|\sigma |\ge  n/\epsilon\,\Big|\,
\Phi, \sigma \in {\cal P}^n_0\,\Bigr\} 
\Pro^0\{\,\sigma \in {\cal P}^n_0\,|\,\Phi\}\\
&\le&\E^0[\calH_{0}^{out,n; W=0}(0)\,|\,\Phi]
\sup_\sigma \Pro^0\Bigl\{\,\sigma \in {\cal P}^n_W,
|\sigma |\ge  n/\epsilon\,\Big|\,
\Phi, \sigma \in {\cal P}^n_0\,\Bigr\}\,,
\end{eqnarray}
where  $\calH_{0}^{out,n;W=0}(0)$ denotes  
the number of paths of length  $n$ 
originating from $(X_0=0,0)$
under the assumption $W=0$.
But
\begin{eqnarray*}
\lefteqn{
\sup_\sigma \Pro^0\Bigl\{\,\sigma \in {\cal P}^n_w,
|\sigma |\ge  n/\epsilon\,\Big|\,
\Phi, \sigma \in {\cal P}^n_0\,\Bigr\}}\\
\\
& \leq & 
\sup_{\sigma=((X_{j_1},1),\ldots,(X_{j_n},n))\atop
{\sum_{i=1}^n |X_{j_i}-X_{j_{i-1}}|\geq n/\epsilon}}
\E^0\Bigl[\prod_{i=1}^n
\delta_{j_{i-1},{j_i}}(i-1,w)\,\Big|\,\Phi,\sigma\in {\cal P}^n_0\Bigr],\,
\end{eqnarray*}
where  $X_{j_0}=0$ and
 $\delta_{j_{i-1},{j_i}}(i-1,w) =\delta_{j_{i-1},{j_i}}(i-1)$ 
is the indicator of the existence of the edge
from   $(X_{j_{i-1}},i-1)$ to $(X_{j_i},i)$
defined  by~(\ref{e.delta}), and where we add
in the notation the dependence on  the noise $W=w$.
Using the conditional independence of marks, (\ref{e.delta}),
(\ref{eq:SINR})  and lack of memory of the exponential
distribution of $F$ of parameter $\mu$ we have for the path-loss
function~(\ref{e.path-loss})
\begin{eqnarray*}
\E^0\Bigl[\prod_{i=1}^n
\delta_{j_{i-1},{j_i}}(i-1,w)\,\Big|\,\Phi,\sigma\in {\cal P}^n_0\Bigr]
&=&\prod_{i=1}^n
\E^0\bigl[\delta_{j_{i-1},{j_i}}(i-1,w)\,\big|\,\Phi,\delta_{j_{i-1},{j_i}}(i-1,0)=1
\bigr]\\
&=&\prod_{i=1}^n\exp 
\left( -\mu (A|X_{j_{i-1}}-X_{j_i}|)^{\beta}{Tw} \right)\,.
\end{eqnarray*}
Hence
$$
\sup_\sigma \Pro^0\Bigl\{\,\sigma \in {\cal P}^n_w,
|\sigma |\ge  n/\epsilon\,\Big|\,
\Phi, \sigma \in {\cal P}^n_0\,\Bigr\}\le
 \exp \left( -\mu A^\beta Tw n \epsilon^{-\beta}\right)\,,
$$
where the last inequality follows from a convexity argument.
Using this and (\ref{eq:choppdefevent2}), we get
\begin{eqnarray*}
\E^0(\Pi_\epsilon(n))&\le&
\E^0[\calH_{0}^{out,n; W=0}(0)]\exp \left( -\mu A^\beta Tw n
  \epsilon^{-\beta}\right)\\
&\le&\xi^n \exp \left( -\mu A^\beta Tw n \epsilon^{-\beta}\right)\\
&\le&\exp \left( n (\log(\xi) - K/\epsilon^{\beta})\right)\,,
\end{eqnarray*}
where in the second inequality we used the following result
of  Corollary~\ref{c.local-finiteness} 
$$\E^0[\calH_{0}^{out,n; W=0}(0)]=h^{out,n; W=0}=
h^{in,k,W=0}\le \xi^k\,$$
and where $K$ is a positive constant.
This shows~(\ref{e.BC}) for $\epsilon$ small enough,
and thus concludes the proof. 
\qed
\end{proof}

\begin{proof}({\em of Corollary~\ref{c.pPoissonInfty}})
Without loss of generality assume $x=0$.
We use the left inequality in~(\ref{e.sandwich}),
(\ref{e.trials_bound}) and the inclusion~(\ref{e.inclusion}) to obtain
$$P_{X(0),X(y)}(0)\ge
\LD_{X(0)}(0)\ge\Trial_{X(0)}(0)\ge\widehat\Trial_{X(0)}(0)
\,$$
and in consequence
$$p(0,y,\Phi)\ge \E[\widehat\Trial_{X(0)}(0)\,|\,\Phi]\,.$$
Using the isotropy 
and the strong Markov property of the Poisson p.p. 
$$\E[\widehat\Trial_{X(0)}(0)\,|\,\Phi]=
\E^0 [\widehat \Trial_0(0)\,|\,\Phi|_{\overline B})]\, ,$$
$\Phi|_{\overline B})$ is the restriction of $\Phi$ to the
complement of the open ball $B=B_{(0,R)}(R)$, centered at $(0,R)$ 
of radius $R\ge0$, where $R$ is r.v. 
independent of $\Phi$ and
having for density 
$$\frac {\md\theta}{2\pi} 2\pi \lambda r \exp(-\lambda \pi r^2)\,.$$ 
But since we consider here the SNR graph $\widehat\Gsinr$
$$\E^0 [\widehat\Trial_0(0)\,|\,\Phi|_{\overline B})] \ge 
\E^0 [ \widehat\Trial_0(0)\,|\,\Phi]\,.$$
The result follows now from  Lemma~\ref{l.multicast-infinite}. 
\qed
\end{proof}

\section{SINR space-time graph and routing}
\label{s.remarks}
Let us now translate our results regarding the SINR graph
into properties of routing in ad-hoc networks.

Firstly, it makes sense to assume that any routing algorithm builds paths on
$\Gsinr$. This takes two key phenomena into account: 
contention for channel (nodes have to wait for some particular time slots
to transmit a packet) and collisions (lack of capture due to insufficient SINR). 

Our time constant gives bounds on the delays that can be attained
in the ad-hoc network by any routing algorithms. 
Of course, realistic routing policies cannot use information about future 
channel conditions. In the case of Poisson p.p. there is hence no
routing algorithm with a finite time constant. The existence of such an
algorithm in the case of the Poisson$+$Grid p.p. remains an open question.
In the Poisson p.p. case; one can ask about the exact asymptotics of
the optimal delay (we know it is not linear)  
and of the delay realizable by some non-anticipating algorithm. 

Let us discuss now the relation of our results to those obtained
in~\cite{net:Ganti09spaswin,JMMR2009}. In these papers the so called
delay-tolerant networks are considered and modeled 
by a spatial SINR or signal-to-noise ratio (SNR) graph with no time dimension. 
In these models, the time constant (defined there as the asymptotic
ratio of the graph distance to the Euclidean distance) is announced
to be finite, even in the pure Poisson case. 
The reason for the different performance of these models lays in the
fact that they do not take the time required for a successful transmission
from a given node in the evaluation of the end-to-end delay.
The heavy-tailness of this time (which follows from that of the exit time
(cf. Proposition~\ref{p.multicast-infinite}) 
makes the time constant infinite in the space-time Poisson scenario.
The reason for the heavy-tailness of the successful transmission
time is linked to the so called ``RESTART'' algorithm (see
e.g.~\cite{JelenkovicTanInfocom07,%
JelenkovicTan2007,Asmussen,JelenkovicTan2009}). 
In our case the spatial irregularities in the ad-hoc network
play a role similar to that of the file size variability in the RESTART scenario.

\bibliographystyle{plain}
\bibliography{fppsinr}

\begin{thebibliography}{10}

\bibitem{Asmussen}
S.~Asmussen, P.~Fiorini, L.~Lipsky, T.~Rolski, and R.~Sheahan.
\newblock Asymptotic behavior of total times for jobs that must start over if a
  failure occurs.
\newblock {\em Mathematics of Operations Research}, 33(4):932--944, 2008.

\bibitem{allerton}
F.~Baccelli, B.~Blaszczyszyn, and P.~M{\"u}hlethaler.
\newblock An {A}loha protocol for multihop mobile wireless networks.
\newblock In {\em Proceedings of the Allerton Conference}, Urbana Champaign,
  Illinois, November 2003.
\newblock and {\em IEEE Transactions on Information Theory}, 52(2):421--436,
  2006.

\bibitem{bbm06}
F.~Baccelli, B.~Blaszczyszyn, and P.~Muhlethaler.
\newblock An aloha protocol for multihop mobile wireless networks.
\newblock {\em IEEE Transactions on Information Theory}, 52(2):421--436, 2006.

\bibitem{JSAC}
F.~Baccelli, B.~B{\l}aszczyszyn, and P.~M\"uhlethaler.
\newblock Stochastic analysis of spatial and opportunistic {A}loha.
\newblock {\em IEEE JSAC, special issue on Stochastic Geometry and Random
  Graphs for Wireless Networks}, 27:1109--1119, 2009.

\bibitem{compj}
F.~Baccelli, B.~B{\l}aszczyszyn, and P.~M\"uhlethaler.
\newblock Time-space opportunistic routing in wireless ad hoc networks,
  algorithms and performance.
\newblock {\em the Computer Journal}, 2009.
\newblock to appear.

\bibitem{RST}
F.~Baccelli and C.~Bordenave.
\newblock The radial spanning tree of a poisson point process.
\newblock {\em Annals of Applied Probab.}, 17(1):305--359, 2007.

\bibitem{BordenaveAaP}
Charles Bordenave.
\newblock Navigation on a poisson point process.
\newblock {\em Ann. Appl. Probab.}, 18:708--746, 2008.

\bibitem{Broad_Hamm57}
Simon Broadbent and John Hammersley.
\newblock Percolation processes {I}. crystals and mazes.
\newblock {\em Proc. Camb. Phil. Soc}, 53:629--–641, 1957.

\bibitem{DVJII2008}
D.~J. Daley and D.~Vere-Jones.
\newblock {\em An Introduction to the Theory of Point Processes, vol. {II}}.
\newblock Springer, New York, 2008.

\bibitem{Dousse_etal_TON}
O~Dousse, F.~Baccelli, and P~Thiran.
\newblock Impact of interferences on connectivity in ad-hoc networks.
\newblock {\em IEEE/ACM Trans. Networking}, 13:425--543, 2005.

\bibitem{Dousse_etal}
O.~Dousse, M.~Franceschetti, N.~Macris, R.~Meester, and P.~Thiran.
\newblock Percolation in the signal to interference ratio graph.
\newblock {\em Journal of {A}pplied {P}robability}, 43(2):552--562, 2006.

\bibitem{net:Ganti09spaswin}
R.~K. Ganti and M.~Haenggi.
\newblock Bounds on information propagation delay in interference-limited
  {ALOHA} networks.
\newblock In {\em Proc. of Workshop on Spatial Stochastic Models for Wireless
  Networks}, 2009.

\bibitem{Gilbert61}
E.~N. Gilbert.
\newblock Random plane networks.
\newblock {\em SIAM J.}, 9:533--543, 1961.

\bibitem{HowardNewman1997}
C.~D. Howard and C.~M. Newman.
\newblock Euclidean models of first-passage percolation.
\newblock {\em Probab. Theory Relat. Fields}, 108:153--170, 1997.

\bibitem{JMMR2009}
P.~Jacquet, B.~Mans, P.~M\"uhlethaler, and G.~Rodolakis.
\newblock Opportunistic routing in wireless ad hoc networks: Upper bounds for
  the packet propagation speed.
\newblock {\em IEEE JSAC, special issue on Stochastic Geometry and Random
  Graphs for Wireless Networks}, 27:1192--1202, 2009.

\bibitem{JelenkovicTan2009}
P.~R. Jelenkovic and J.~Tan.
\newblock Stability of finite population aloha with variable packets.
\newblock Technical Report arXiv:0902.4481v2, 2009.

\bibitem{JelenkovicTanInfocom07}
P.R. Jelenkovi\'c and J.~Tan.
\newblock Can retransmissions of superexponential documents cause
  subexponential delays?
\newblock In {\em Proc. of INFOCOM}, Anchorage, AL, USA, 2007. IEEE.

\bibitem{JelenkovicTan2007}
P.R. Jelenkovi\'c and J.~Tan.
\newblock Is aloha causing power law delays?
\newblock In L.~Mason, T.~Drwiega, and J.~Yan, editors, {\em Managing Traffic
  Performance in Converged Networks}. Springer, Berlin, 2007.

\bibitem{Kingman73}
J.F.C. Kingman.
\newblock Subadditive ergodic theory.
\newblock {\em Annals of Probab.}, 1:883--899, 1973.

\bibitem{Kleinberg}
J.M. Kleinberg.
\newblock The small-world phenomenon: an algorithmic perspective.
\newblock In {\em Proc. 32nd Annual ACM Symposium on the Theory of Computing},
  pages 163--170, 2000.

\bibitem{Pimentel2006}
L.P.R. Pimentel.
\newblock The time constant and critical probabilities in percolation models.
\newblock {\em Elect. Comm. in Probab.}, 11:160--168, 2006.

\bibitem{PughShub71}
Ch. Pugh and M.~Shub.
\newblock Ergodic elements of ergodic actions.
\newblock {\em Compositio Mathematica}, 23(1):115--122, 1971.

\bibitem{TseViswanath2005}
D.~Tse and P.~Viswanath.
\newblock {\em Foundamentals of Wireless Communication}.
\newblock Cambridge University Press, 2005.

\bibitem{VahidiAsl_Wierman90}
M.~Q. Vahidi-Asl and J.~C. Wierman.
\newblock First-passage percolation on the {V}oronoi tessellation and
  {D}elaunay trangulation.
\newblock In M.~Karo\'ski, J.~Jaworski, and A.~Ruci\'nski, editors, {\em Random
  Graphs'87; Based on Proceedings of the 3rd International Seminar on Random
  Graphs and Probabilistic Methods in Combinatorics, June~27 -- July~3}, pages
  341--359. John Wiley \& sons, Chichester, 1990.

\end{thebibliography}

\end{document}